\documentclass[12pt]{amsart}
\usepackage[utf8]{inputenc}
\usepackage{mathrsfs}
\usepackage{amssymb}
\usepackage{bm}
\usepackage{graphicx}
\usepackage[centertags]{amsmath}
\usepackage{amsfonts}
\usepackage{amsthm}
\usepackage{enumerate}
\usepackage{tocvsec2}
\usepackage{xcolor}
\usepackage[margin=1in]{geometry}

\newtheorem{theorem}{Theorem}[section]
\newtheorem*{ta}{Theorem 1}
\newtheorem*{tb}{Theorem 2}
\newtheorem*{corollary*}{Main Corollary}

\newtheorem{corollary}[theorem]{Corollary}
\newtheorem{lemma}[theorem]{Lemma}
\newtheorem{rem}[theorem]{Remark}

\newtheorem{proposition}[theorem]{Proposition}

\newtheorem{claim}[theorem]{Claim}

\newtheorem{question}{Question}[section]
\theoremstyle{definition}

\newcommand{\rr}{\mathbb{R}}
\newcommand{\nn}{\mathbb{N}}

\newcommand{\ee}{\varepsilon}

\newcommand{\ddd}{\mathcal{D}}

\newcommand{\hhh}{\mathcal{H}}

\newcommand{\kkk}{\mathcal{K}}

\newcommand{\ttt}{\mathcal{T}}
\newcommand{\eee}{\mathcal{E}}
\newcommand{\nnn}{\mathcal{N}}

\newcommand{\cat}{^\smallfrown}

\begin{document}

\title[Szlenk and $w^*$-dentability indices]{A note on the relationship between the  Szlenk \\ and $w^*$-dentability indices  of arbitrary $w^*$-compact sets}
\author{R.M. Causey}

\begin{abstract} We prove the optimal estimate between the Szlenk and $w^*$-dentability indices of an arbitrary $w^*$-compact subset of the dual of a Banach space.  For a given $w^*$-compact, convex subset $K$ of the dual of a Banach space, we introduce a two player game the winning strategies of which determine the Szlenk index of $K$. We give applications to the $w^*$-dentability index of a Banach space and of an operator.

\end{abstract}

\maketitle

\section{Introduction}

Since its inception in \cite{Szlenk}, the Szlenk index has been an important tool in renorming theory \cite{KOS}, \cite{R}, \cite{GKL}.  In \cite{CD}, the notion of $\xi$-asymptotically uniformly smooth operators was given, with the $0$-asymptotically uniformly smooth notion generalizing the notion of an asymptotically uniformly smooth Banach space.  It was shown in \cite{CD} that an operator $A:X\to Y$ has Szlenk index not exceeding $\omega^{\xi+1}$ if and only if there exists an equivalent norm $|\cdot|$ on $Y$ making $A:X\to (Y, |\cdot|)$ $\xi$-asymptotically uniformly smooth.  Applying this to the identity of a Banach space, we deduce that a Banach space $X$ has Szlenk index not exceeding $\omega^{\xi+1}$ if and only if there exists an equivalent norm $|\cdot|$ on $X$ such that $(X, |\cdot|)$ is $\xi$-asymptotically uniformly smooth.  

Another index has been used to study the class of Asplund spaces, the $w^*$-dentability index.  The $w^*$-dentability index is distinct from the Szlenk index, but each characterizes $w^*$-fragmentability of a $w^*$-compact set.  Since both indices characterize $w^*$-fragmentability, it is natural to ask what relationship must exist between the indices. It follows immediately from the definitions that the Szlenk index of a set cannot exceed its $w^*$-dentability index. We discuss in the next section the different results obtained in the literature regarding the relationship between the $w^*$-dentability and Szlenk indices.  

In what follows, $Sz(K)$ (resp. $Dz(K)$) will denote the Szlenk (resp.  $w^*$-dentability index) of the set $K$.

\begin{ta} Let  $X$ be a Banach space, let $K\subset X^*$ be $w^*$-compact, and let $\xi$ be an ordinal. \begin{enumerate}[(i)]\item If $Sz(K)\leqslant \omega^\xi$, then $Dz(K)\leqslant \omega^{1+\xi}$. \item Suppose that  $K$ is convex. Then  $Dz(K)\leqslant \omega Sz(K)$, and if $Sz(K)\geqslant \omega^\omega$, $Dz(K)=Sz(K)$. \end{enumerate} 
\label{main theoremm}
\end{ta}

As was discussed in \cite{HS}, for every $n\in \nn\cup \{0\}$, there exist Banach spaces $X_n$, $Y_n$ such that $Sz(X_n)=Sz(Y_n)=\omega^n$ while $Dz(X_n)=\omega^{n+1}$ and $Dz(Y_n)=\omega^n$.     These examples show the sharpness of Theorem \ref{main theoremm}.

In \cite{AJO}, it was shown that one can compute the Szlenk index of a separable Banach space containing no isomorph of $\ell_1$ by considering convex combinations of the branches of trees of vectors satisfying a certain weak nullity condition.  We also recall a particular two-player game played on a Banach space.  For $\ee>0$ and every $n\in \nn$, Player I chooses a subspace $Z_1^n$ of $X$ such that $\dim(X/Z_1^n)<\infty$, Player II chooses a vector $x_1^n\in B_{Z_1^n}$, $\ldots$, Player I chooses a subspace $Z_n^n$ of $X$ such that $\dim(X/Z_n^n)<\infty$, and Player II chooses a vector $x_n\in B_{Z_n^n}$.  We say that Player II wins the game if for every $n\in \nn$, $\|n^{-1}\sum_{i=1}^n x_i^n\|\geqslant \ee$, and Player I wins otherwise.  Then if $X$ is a separable Banach space not containing $\ell_1$, the results of \cite{AJO} combined with the results of \cite{KOS} imply that $Sz(X)\leqslant \omega$ if and only if for every $\ee>0$, Player I has a winning strategy in this game.  Since this game is determined, $Sz(X)>\omega$ if and only if for some $\ee>0$, Player II has a winning strategy in this game.  Note that we require a certain ``smallness'' condition on a \emph{specific} convex combination $n^{-1}\sum_{i=1}^n x_i^n$ of $(x_i^n)_{i=1}^n$.

In \cite{C}, the results of \cite{AJO} were extended to allow one to compute the Szlenk index of an arbitrary $w^*$-compact subset of the dual of an arbitrary Banach space.  In analogy to the game defined above, we wish to define for a given ordinal $\xi$ a certain game the winning strategies of which determine whether the Szlenk index of an arbitrary $w^*$-compact set exceeds $\omega^\xi$.  Given a Banach space $X$, let $\ddd$ denote the subspaces of $X$ having finite codimension in $X$, and let $\kkk$ denote the norm-compact subsets of $X$. Let $K\subset X^*$ be $w^*$-compact.  Suppose that $\Lambda$ is a set, $T$ is a non-empty collection of non-empty sequences in $\Lambda$ such that there does not exist an infinite sequence $(\zeta_i)_{i=1}^\infty\subset \Lambda$ all the finite initial segments of which lie in $T$ (such a collection $T$ is called a \emph{non-empty, well-founded} $B$-\emph{tree}). Assume also that $\mathbb{P}:T\to \rr$ is a fixed function.  For $\ee>0$, we let Player I choose $Z_1\in \ddd$ and $\zeta_1\in \Lambda$ such that  $(\zeta_1)\in T$.  Player II then chooses $C_1\in \kkk$.  Next, assuming $(\zeta_i)_{i=1}^n\in T$, $Z_1, \ldots, Z_n\in \ddd$, and $C_1, \ldots, C_n\in \kkk$ have been chosen, if $(\zeta_i)_{i=1}^n$ has no proper extensions in $T$, the game terminates.  Otherwise Player I chooses $\zeta_{n+1}\in \Lambda$ such that $(\zeta_i)_{i=1}^{n+1}\in T$ and $Z_{n+1}\in \ddd$.  Player II chooses $C_{n+1}\in \kkk$.  Our assumptions on $T$ yield that this game must terminate after finitely many turns.  Let us assume the game terminates with the choices $(\zeta_i)_{i=1}^n$, $(Z_i)_{i=1}^n$, $(C_i)_{i=1}^n$.  We  say that Player II wins the game if there exist a sequence $(x_i)_{i=1}\in \prod_{i=1}^n( B_X\cap Z_i\cap C_i)$ and $x^*\in K$ such that $$\text{Re\ }x^*\Bigl(\sum_{i=1}^n \mathbb{P}((\zeta_j)_{j=1}^i)x_i\Bigr)\geqslant \ee,$$ and let us say Player I wins otherwise.  Let us refer to this as the $(\ee, K, \mathbb{P})$ game on $T.\ddd.\kkk$.   Our main result in this direction is the following.

\begin{tb} For every ordinal $\xi$, there exists a non-empty, well-founded $B$-tree $\Gamma_\xi$ on $[0, \omega^\xi]$ and a function $\mathbb{P}_\xi:\Gamma_\xi\to \rr$ such that for any Banach space $X$ and any $w^*$-compact $K\subset X^*$, $Sz(K)>\omega^\xi$ if and only if there exists $\ee>0$ such that Player II has a winning strategy in the $(\ee, K, \mathbb{P}_\xi)$-game on $\Gamma_\xi.\ddd.\kkk$, and $Sz(K)\leqslant \omega^\xi$ if and only if for every $\ee>0$, Player I has a winning strategy in the $(\ee, K, \mathbb{P}_\xi)$ game on $\Gamma_\xi.\ddd.\kkk$.

\end{tb}

\section{Definitions}

\subsection{Definition of the indices}
Let $X$ be a Banach space and let $K\subset X^*$.  For $\ee>0$, we let $s_\ee(K)$ denote those $x^*\in K$ such that for every $w^*$-neighborhood $V$ of $x^*$, $\text{diam}(V\cap K)>\ee$. We let $d_\ee(K)$ denote those $x^*\in K$ such that for every $w^*$-open slice $S$ containing $x^*$, $\text{diam}(S\cap K)>\ee$.  Recall that a $w^*$-open slice is a subset of $X^*$ of the form $\{y^*: \text{Re\ }y^*(x)>a\}$ for some $x\in X$ and $a\in \rr$.   We then define $s_\ee^0(K)=K$, $s_\ee^{\xi+1}(K)=s_\ee(s_\ee^\xi(K))$, and $s^\xi_\ee(K)=\cap_{\zeta<\xi}s_\ee^\zeta(K)$ when $\xi$ is a limit ordinal.  We set $Sz(K, \ee)=\min\{\xi: s_\ee^\xi(K)=\varnothing\}$ if this class of ordinals is non-empty, and we set $Sz(K, \ee)=\infty$ otherwise. We let $Sz(K)=\sup_{\ee>0}Sz(K, \ee)$, where we agree that $\xi<\infty$ for all ordinals $\xi$. If $X$ is a Banach space, we let $Sz(X, \ee)=Sz(B_{X^*}, \ee)$ and $Sz(X)=Sz(B_{X^*})$.  If $A:X\to Y$ is an operator, we let $Sz(A, \ee)=Sz(A^*B_{Y^*}, \ee)$ and $Sz(A)=Sz(A^*B_{Y^*})$.   We define $d_\ee^\xi(K)$, $Dz(K, \ee)$, $Dz(K)$, etc.,  similarly. It is quite clear that $Sz(K)\leqslant Dz(K)$.    

We recall that $K$ is said to be $w^*$-\emph{fragmentable} provided that for every non-empty subset $L$ of $K$ and every $\ee>0$, there exists a $w^*$-open subset $U$ of $X^*$ such that $U\cap L\neq \varnothing$ and $\text{diam}(U\cap L)<\ee$.  We say that $K$ is $w^*$-\emph{dentable} if for any non-empty subset $L$ of $K$ and every $\ee>0$, there exists a $w^*$-open slice $S$ of $X^*$ such that $S\cap L\neq \varnothing$ and $\text{diam}(S\cap L)<\ee$.   It is clear that $K$ is $w^*$-fragmentable (resp. $w^*$-dentable) if and only if $Sz(K)$ (resp. $Dz(K)$) is an ordinal.  Moreover, $w^*$-fragmentability and $w^*$-dentability are equivalent, which is a consequence of Theorem \ref{main theoremm}.  Since these properties are equivalent, it is natural to consider the relationship between $Sz(K)$ and $Dz(K)$. Lancien \cite{L0} proved using descriptive set theoretic techniques that there exists a function $\phi:[0, \omega_1)\to [0, \omega_1)$ such that if $\xi<\omega_1$ and if $X$ is a Banach space with $Sz(X)\leqslant \xi$, $Dz(X)\leqslant \phi(\xi)$.  Raja \cite{R} proved that for any Banach space (without assumption of countability of $Sz(X)$) that $Dz(X)\leqslant \omega^{Sz(X)}$.    H\'{a}jek and Schlumprecht \cite{HS} showed that if $Sz(X)$ is countable, $Dz(X)\leqslant \omega Sz(X)$.   The content of Theorem \ref{main theoremm} extends this result of H\'{a}jek and Schlumprecht to the general case of an arbitrary $w^*$-compact, convex set $K$ as opposed to the case $K=B_{X^*}$, and removes the hypothesis of countability of $Sz(K)$. 

We note that the most interesting case, of course, is the case $K=B_{X^*}$.  However, the case $K=A^*B_{Y^*}$ for an operator $A:X\to Y$ is also of interest.  We refer the reader to \cite{Brooker}, \cite{CD}, and \cite{Cpower} for results concerning the Szlenk index of an operator, including renorming theorems for asymptotically uniformly smooth operators.  However, to our knowledge,  the $w^*$-dentability index of an operator has not been investigated.   

\subsection{$B$-trees}

Given a set $\Lambda$, we let $\Lambda^{<\nn}$ denote the finite sequences in $\Lambda$, including the empty sequence, $\varnothing$. We write $s\preceq t$ if $s$ is an initial segment of $t$.  If $t\in \Lambda^{<\nn}$, we let $|t|$ denote the length of $t$ and for $0\leqslant i\leqslant |t|$, $t|_i$ is the initial segment of $t$ having length $i$. If $\varnothing\neq t$, we let $t^-=t|_{|t|-1}$. We let $s\cat t$ denote the concatenation of $s$ and $t$. A subset $T$ of $\Lambda^{<\nn}$ is called a \emph{tree} if for all $t\in T$ and $s\preceq t$, $s\in T$.  A subset $T$ of $\Lambda^{<\nn}\setminus \{\varnothing\}$ will be called a $B$-\emph{tree} provided that for any $t\in T$ and any $\varnothing\prec s\preceq t$, $s\in T$.   We let $MAX(T)$ denote the members of $T$ which are $\prec$-maximal and $T'=T\setminus MAX(T)$.  We define $T^0=T$, $T^{\xi+1}=(T^\xi)'$, and $T^\xi=\cap_{\zeta<\xi}T^\zeta$ when $\xi$ is a limit ordinal.  We say $T$ is \emph{well-founded} if there exists an ordinal $\xi$ such that $T^\xi=\varnothing$, and we let $o(T)$ denote the smallest such $\xi$.  If no such $\xi$ exists, we say $T$ is \emph{ill-founded} and write $o(T)=\infty$.  Note that $o(T)=\infty$ if and only if there exists an infinite sequence $(\zeta_i)_{i=1}^\infty\subset \Lambda$ such that $(\zeta_i)_{i=1}^n\in T$ for all $n\in \nn$. 

Recall that for any $B$-trees $S,T$, a function $\theta:S\to T$ is called \emph{monotone} provided that for any $\varnothing\prec s\prec s_1\in S$, $\theta(s)\prec \theta(s_1)$.  

Given non-empty sets $\Lambda_1, \ldots, \Lambda_k$, we identify the set $(\prod_{i=1}^k \Lambda_i)^{<\nn}$ with the set $\{(t_i)_{i=1}^k\in \prod_{i=1}^k \Lambda_i^{<\nn}: |t_1|=\ldots =|t_k|\}$.  The identification is obtained by identifying $\varnothing$ with $(\varnothing, \ldots, \varnothing)$ and, for $n>0$, $$\bigl((a_{1i}, \ldots, a_{ki})\bigr)_{i=1}^n\leftrightarrow \bigl((a_{1i})_{i=1}^n, (a_{2i})_{i=1}^n, \ldots, (a_{ki})_{i=1}^n\bigr).$$

Let $X$ be a Banach space and let $T$ be a $B$-tree.  Let us say that a collection $(x_t)_{t\in T}\subset X$ is \emph{weakly null} provided that for every ordinal $\xi$, every $t\in (T\cup \{\varnothing\})^{\xi+1}$, and every $Z\leqslant X$ with $\dim(X/Z)<\infty$, there exists $s\in T^\xi$ with $s^-=t$ such that $x_s\in Z$.

We last define some $B$-trees which will be important for us.  If $(\zeta_i)_{i=1}^n$ is a sequence of ordinals and $\zeta$ is an ordinal, we let $\zeta+(\zeta_i)_{i=1}^n=(\zeta+\zeta_i)_{i=1}^n$.  If $G$ is a collection of non-empty sequences of ordinals and $\zeta$ is an ordinal, we let $\zeta+G=\{\zeta+t: t\in G\}$.  We let $$\ttt_0=\varnothing,$$ $$\ttt_{\xi+1}=\{(\xi+1)\cat t: t\in \{\varnothing\}\cup \ttt_\xi\},$$ and if $\xi$ is a limit ordinal, we let $$\ttt_\xi=\underset{\zeta<\xi}{\bigcup}\ttt_{\zeta+1}.$$  Note that this union is a totally incomparable union.   For each ordinal $\xi$, $\ttt_\xi$ is a $B$-tree on $[0, \xi]$ with $o(\ttt_\xi)=\xi$.

Next, let $$\Gamma_0=\{(1)\},$$ \begin{align*} \Gamma_{\xi+1}=\Bigl\{(\omega^\xi(n-1)+t_1)\cat \ldots \cat (\omega^\xi(n-m)+t_m): &  n\in \nn, 1\leqslant m\leqslant n, t_i\in \Gamma_\xi,  \\ & t_i\in MAX(\Gamma_\xi)\text{\ for each}1\leqslant i<m\Bigr\},\end{align*} and when $\xi$ is a limit ordinal, $$\Gamma_\xi=\underset{\zeta<\xi}{\bigcup} (\omega^\zeta+\Gamma_{\zeta+1}).$$   For each ordinal $\xi$, $\Gamma_\xi$ is a $B$-tree on $[1, \omega^\xi]$ with $o(\Gamma_\xi)=\omega^\xi$.   We define $\mathbb{P}_\xi:\Gamma_\xi\to [0,1]$ by letting $\mathbb{P}_0((1))=1$, $$\mathbb{P}_{\xi+1}((\omega^\xi(n-1)+t_1)\cat \ldots \cat (\omega^\xi(n-m)+t_m))=\mathbb{P}_\xi(t_m)/n,$$ and $$\mathbb{P}_\xi(\omega^\zeta+t)=\mathbb{P}_{\zeta+1}(t), \hspace{2mm} t\in \Gamma_{\zeta+1}.$$   We refer the reader to \cite{C2} for a discussion that these functions are well-defined and for every ordinal $\xi$ and every $t\in MAX(\Gamma_\xi)$, $\sum_{s\preceq t}\mathbb{P}_\xi(s)=1$.

\section{Games on well-founded $B$-trees}

Given a non-empty, well-founded $B$-tree $T$ on the set $\Lambda$, let $R_T=\{\zeta\in \Lambda: (\zeta)\in T\}$.    Given a non-empty, well-founded $B$-tree $T$ and two non-empty sets $ \ddd, \kkk$, we let $T.\ddd.\kkk$ denote the sequences $(\zeta_i, Z_i, C_i)_{i=1}^n$ such that $Z_i\in \ddd$, $C_i\in \kkk$, $(\zeta_i)_{i=1}^n\in T$.  Let $T.\ddd=\{(\zeta_i,Z_i)_{i=1}^n: Z_i\in \ddd, (\zeta_i)_{i=1}^n\in T\}$. Note that $T.\ddd.\kkk$ and $T.\ddd$ are non-empty, well-founded $B$-trees with the same order as $T$.      Given a subset $\eee\subset MAX(T.\ddd.\kkk)$, we define the $\eee$-game on $T.\ddd.\kkk$ as follows: Player I chooses $Z_1\in \ddd$ and $\zeta_1\in R_T$.  Player II chooses $C_1\in \kkk$.  Next, assuming that $Z_1, \ldots, Z_n\in \ddd$, $C_1, \ldots, C_n\in \kkk$, and $\zeta_1, \ldots, \zeta_n\in \Lambda$ have been chosen such that $(\zeta_i)_{i=1}^n\in T$, if $(\zeta_i)_{i=1}^n\in MAX(T)$, the game terminates.  Otherwise Player I chooses $Z_{n+1}\in \ddd$ and $\zeta_{n+1}\in \Lambda$ such that $(\zeta_i)_{i=1}^{n+1}\in T$ and player II chooses $C_{n+1}\in \kkk$.  Since $T$ is well-founded, the game terminates after some finite number of steps.  Suppose that the game terminates after the choices $C_1, \ldots, C_n\in \kkk$, $Z_1, \ldots, Z_n\in \ddd$, and $\zeta_1, \ldots, \zeta_n\in \Lambda$.   Then Player I wins provided $(\zeta_i, Z_i, C_i)_{i=1}^n\in MAX(T.\ddd.\kkk)\setminus \eee$, and Player II wins if $(\zeta_i, Z_i, C_i)_{i=1}^n\in \eee$.  We call such a game a \emph{game on a non-empty, well-founded} $B$-\emph{tree}.  

A \emph{strategy for Player} I is a function $\varphi:T'.\ddd.\kkk\cup \{\varnothing\}\to \Lambda\times \ddd$ such that if $\varphi(\varnothing)=(\zeta, Z)$, $\zeta\in R_T$, and if $\varphi((\zeta_i, Z_i, C_i)_{i=1}^n)=(\zeta_{n+1}, Z_{n+1})$, $(\zeta_i)_{i=1}^{n+1}\in T$.  A sequence $(\zeta_i, Z_i, C_i)_{i=1}^n\in MAX(T.\ddd.\kkk)$ is  $\varphi$-\emph{admissible} if $(\zeta_j, Z_j)=\varphi((\zeta_i, Z_i)_{i=1}^{j-1})$ for each $1\leqslant j\leqslant n$.   A strategy for Player I $\varphi$ is called a \emph{winning strategy for the} $\eee$ \emph{game on } $T.\ddd.\kkk$ provided that every $\varphi$-admissible sequence lies in $MAX(T.\ddd.\kkk)\setminus \eee$.  A \emph{winning substrategy for Player} I \emph{for the} $\eee$\emph{game on} $T.\ddd.\kkk$ is a subset $S$ of $T'.\ddd.\kkk\cup \{\varnothing\}$ containing $\varnothing$ and a function $\phi:S\to \Lambda\times \ddd$ such that, if $(\zeta, Z)=\phi(\varnothing)$,  \begin{enumerate}[(i)]\item $S=\{\varnothing\}\cup \{t\in T'.\ddd.\kkk: (\exists C\in \kkk)((\zeta, Z,C)\preceq t)\}$, \item $\zeta\in R_T$, \item if $t=(\zeta_i, Z_i, C_i)_{i=1}^n\in S$ and $(\zeta_{n+1}, Z_{n+1})=\phi(t)$, then $(\zeta_i)_{i=1}^{n+1}\in T$, \item if $(\zeta_i, Z_i, C_i)_{i=1}^n\in MAX(T.\ddd.\kkk)$, $(\zeta_1, Z_1)=(\zeta, Z)$,  and $(\zeta_j, Z_j)=\phi((\zeta_i, Z_i, C_i)_{i=1}^{j-1})$ for each $1\leqslant j\leqslant n$, then $t\notin \eee$.  \end{enumerate}

Note that if Player I has a winning substrategy for the $\eee$ game on $T.\ddd.\kkk$, then Player I has a winning strategy.  Indeed, given a winning substrategy $\phi:S\to \Lambda\times \ddd$, we fix any $Z'\in \ddd$ and define a strategy $\varphi:T'.\ddd.\kkk\to \Lambda\times \ddd$ by letting $\varphi|_S=\phi$ and, if $t=(\zeta_i, Z_i, C_i)_{i=1}^n\in T'.\ddd.\kkk\setminus S$, we let $\varphi(t)=(\zeta_{n+1}, Z')$ for any $\zeta_{n+1}\in \Lambda$ such that $(\zeta_i)_{i=1}^{n+1}\in T$.  Such a $\zeta_{n+1}$ exists since $(\zeta_i)_{i=1}^n\in T'$. Let $(\zeta, Z)=\phi(\varnothing)$.  It is straightforward to verify that this is a strategy for Player I.  Since any $\varphi$-admissible sequence $(\zeta_i, Z_i, C_i)_{i=1}^n$ satisfies $(\zeta_1, Z_i)=(\zeta, Z)$, property (iv) of winning substrategy guarantees that $(\zeta_i, Z_i, C_i)_{i=1}^n\in MAX(T.\ddd.\kkk)\setminus \eee$.

A \emph{strategy for Player} II is a $\kkk$-valued function $\psi$  on the set of all pairs $(t,(\zeta_{n+1}, Z_{n+1}))$ such that $t\in \{\varnothing\}\cup T.\ddd.\kkk$, $(\zeta_{n+1}, Z_{n+1})\in \Lambda\times \ddd$, and if $t=(\zeta_i, Z_i, C_i)_{i=1}^n$, $(\zeta_i)_{i=1}^{n+1}\in T$.  A sequence $(\zeta_i, Z_i, C_i)_{i=1}^n\in MAX(T.\ddd.\kkk)$ is $\psi$-\emph{admissible} provided that for every $1\leqslant k\leqslant n$, $\psi((\zeta_i, Z_i, C_i)_{i=1}^{k-1}, (\zeta_k, Z_k))=C_k$.  A strategy for player II $\psi$ is called a \emph{winning strategy for the} $\eee$ \emph{game on } $T.\ddd.\kkk$ provided that every $\psi$-admissible sequence lies in $\eee$. Obviously for a given subset $\eee$ of $MAX(T.\ddd.\kkk)$, Player I and Player II cannot both have a winning strategy.  

\begin{proposition} Every game on a non-empty, well-founded $B$-tree is determined.  That is, exactly one of Player I and Player II has a winning strategy.  
\label{determined}
\end{proposition}

\begin{proof} We prove by induction on $\xi\geqslant 1$ that if $T$ is a non-empty, well-founded $B$-tree with $o(T)\leqslant \xi$, then for any $\eee\subset MAX(T.\ddd.\kkk)$, either Player I has a winning strategy or Player II has a winning strategy.  Assume that for some ordinal $\xi$ and every $1\leqslant \gamma<\xi$, the statement is true hypothesis is true for $\gamma$.  Let $T$ be a non-empty, well-founded $B$-tree with $o(T)=\xi$.  For every $\zeta\in R_T$, let $T(\zeta)$ denote those non-empty sequences $t$ such that $(\zeta)\cat t\in T$.  Note that $T(\zeta)$ is a $B$-tree with $o(T(\zeta))<\xi$, and $T(\zeta)= \varnothing$ if and only if $\zeta\in MAX(T)$.   Given $\zeta\in R_T$, $Z\in \ddd$, and $C\in \kkk$, let $\eee(\zeta, Z, C)$ denote those non-empty sequences $(\zeta_i, Z_i, C_i)_{i=1}^n$ such that $(\zeta, Z, C)\cat (\zeta_i, Z_i, C_i)_{i=1}^n\in \eee$.   Let $W$ denote the set of those  $(\zeta, Z)\in T.\ddd$ such that either \begin{enumerate}[(i)]\item $\zeta\in MAX(T)$ and for every $C\in \kkk$, $(\zeta, Z, C)\in MAX(T.\ddd.\kkk)\setminus \eee$, or \item $\zeta\notin MAX(T)$ and for every $C\in \kkk$, Player I has a winning strategy in the $\eee(\zeta,Z,K)$ game on $T(\zeta).\ddd.\kkk$. 

\end{enumerate}

By the inductive hypothesis, if $(\zeta, Z)\in T.\ddd\setminus W$, then either \begin{enumerate}[(i)]\item $\zeta\in MAX(T)$ and there exists $C\in \kkk$ such that $(\zeta, Z, C)\in \eee$, or \item $\zeta\notin MAX(T)$ and there exists $C\in \kkk$ such that Player II has a winning strategy in the $\eee(\zeta, Z, C)$ game on $T(\zeta).\ddd.\kkk$. \end{enumerate}

It is obvious that Player I has a winning strategy in the $\eee$ game on $T.\ddd.\kkk$ if $W\neq \varnothing$, and Player II has a winning strategy in the $\eee$ game on $T.\ddd.\kkk$ if $W=\varnothing$. For completeness, we define the strategies in each case.

Suppose $W\neq \varnothing$.  Fix $(\zeta, Z)\in W$ and let $S=\{\varnothing\}\cup \{t\in T'.\ddd.\kkk:(\exists C\in \kkk)((\zeta, Z, C)\preceq t)\}$. If $\zeta\in MAX(T)$, then we define $\phi(\varnothing)=(\zeta, Z)$. Next, suppose $\zeta\notin MAX(T)$. For each $C\in \kkk$, fix a winning strategy $\varphi_C:T(\zeta)'.\ddd.\kkk\to \Lambda\times \ddd$ in the $\eee(\zeta, Z, K)$ game on $T(\zeta).\ddd.\kkk$.  Let $\phi(\varnothing)=(\zeta, Z)$ and for each $C\in \kkk$ and each extension $s=(\zeta, Z,C)\cat t\in T'.\ddd.\kkk$ of $(\zeta, Z, C)$, let $\phi(s)=\phi_C(t)$.  In either case, we have produced a winning subtrategy, which we may extend to a winning strategy by the remarks preceding the proposition.

Next, suppose $W=\varnothing$. Fix $C'\in \kkk$. Fix $(\zeta, Z)\in R_T\times \ddd$. If $(\zeta)\in MAX(T)$,  fix $C_{\zeta, Z}\in K$ such that $(\zeta, Z, C_{\zeta, Z})\in \eee$ and let $\psi(\varnothing, (\zeta, Z))=C_{\zeta, Z}$. If $(\zeta)\in T'$, let $\psi(\varnothing, (\zeta, Z))=C_{\zeta, Z}$, where $C_{\zeta, Z}\in \kkk$ is such that Player II has a winning strategy in the $\eee(\zeta, Z, C_{\zeta, Z})$ game on $T(\zeta).\ddd.\kkk$, and let $\psi_{\zeta, Z}$ be a winning strategy on the appropriate domain.  For $s=(\zeta, Z, C)\cat (\zeta_i, Z_i, C_i)_{i=1}^n$ and $(\zeta_{n+1}, Z_{n+1})\in \Lambda\times \ddd$ such that $(\zeta, \zeta_1, \ldots, \zeta_{n+1})\in T$, let $\psi(s, (\zeta_{n+1}, Z_{n+1}))=C'$ if $C\neq C_{\zeta, Z}$ and $\psi(s, (\zeta_{n+1}, Z_{n+1}))=\psi_{\zeta, Z}((\zeta_i, Z_i, C_i)_{i=1}^n, (\zeta_{n+1}, Z_{n+1}))$ if $C=C_{\zeta, Z}$. This defines a winning strategy for Player II.

\end{proof}

\section{Szlenk games}

In Sections $3$,$4$, and $5$, $X$ will be a fixed Banach space, $\ddd$ will the subspaces of $X$ having finite codimension in $X$, and $\kkk$ will denote set of norm compact subsets of $X$.  Given a non-empty, well-founded $B$-tree $T$ and a collection $(x_{(s,t)})_{(s,t)\in \Pi(T.\ddd)}\subset X$, we say the collection is \emph{normally weakly null} provided that for any $s=(\zeta_i, Z_i)_{i=1}^n\in T.\ddd$ and any $t$ such that $(s,t)\in \Pi(T.\ddd)$, $x_{(s,t)}\in B_{Z_n}$.   We will also use normally weakly null to describe a collection $(x_s)_{s\in T.\ddd}$ such that if $s=(\zeta_i, Z_i)_{i=1}^n\in T.\ddd$, $x_s\in Z_i$.  This is a special case of the previous definition in which the collection $(x_{(s,t)})_{(s,t)\in \Pi(T.\ddd)}$ is such that $x_{(s,t)}$ is independent of $t$.

\subsection{Determination of Szlenk index by games}

Given $K\subset X^*$, $\ee\in \rr$, a $B$-tree $T$, and a function $\mathbb{P}:T\to \rr$, we let $\eee_{K, \ee}(T.\ddd.\kkk, \mathbb{P})$ denote those $(\zeta_i, Z_i, C_i)_{i=1}^n\in MAX(T.\ddd.\kkk)$ such that there exist $x^*\in K$ and $(x_i)_{i=1}^n\in \prod_{i=1}^n (B_X\cap Z_i\cap C_i)$ such that  $$\text{Re\ }x^*\Bigl(\sum_{i=1}^n\mathbb{P}((\zeta_j)_{j=1}^i)x_i\Bigr)\geqslant \ee.$$

Given a function $\mathbb{P}:T\to \rr$, we will consider the function $\mathbb{P}$ to be also defined on $T.\ddd$ by $\mathbb{P}((\zeta_i, Z_i)_{i=1}^n)=\mathbb{P}((\zeta_i)_{i=1}^n)$. 

\begin{lemma} Fix a non-empty, well-founded $B$-tree $T$, a function $\mathbb{P}:T\to \rr$, $\ee\in \rr$, and a subset $K$ of $X^*$. If Player II has a winning strategy in the $\eee_{K, \ee}(T.\ddd.\kkk, \mathbb{P})$ game, then there exist a normally weakly null collection $(x_{(s,t)})_{(s,t)\in \Pi (T.\ddd)}\subset B_X$, a collection $(x^*_t)_{t\in MAX(T.\ddd)}\subset K$, and a collection $(C_s)_{s\in T.\ddd}\subset \kkk$ such that \begin{enumerate}[(i)]\item for every $t\in MAX(T.\ddd)$, $$\text{\emph{Re}\ }x^*_t\Bigl(\sum_{s\preceq t} \mathbb{P}(s)x_{(s,t)}\Bigr)\geqslant \ee,$$ and \item for every $s\in T.\ddd$ and any maximal extension $t\in T.\ddd$ of $s$, $x_{(s,t)}\in C_s$. \end{enumerate}

\label{finish}

\end{lemma}

\begin{proof} Fix a winning strategy $\psi$ for Player II in the $\eee_{K, \ee}(T.\ddd.\kkk, \mathbb{P})$ game.  We first define $C_s\in \kkk$ for $s\in T.\ddd$ by induction on $|s|$.   If $|s|=1$, write $s=(\zeta, Z)$ and let $C_s=\psi(\varnothing, (\zeta, Z))$.   Next, suppose that for some $j\in \nn$ and some sequence $s=(\zeta_i, Z_i)_{i=1}^{j+1}\in T.\ddd$, $C_{s|_i}$ has been defined for each $1\leqslant i\leqslant j$.  Let $C_s=\psi((\zeta_i, Z_i, C_{s|_i})_{i=1}^j, (\zeta_{j+1}, Z_{j+1}))$.  This completes the definition of $(C_s)_{s\in T.\ddd}$.  Note that with this definition, for every $t=(\zeta_i, Z_i)_{i=1}^n\in MAX(T.\ddd)$, the sequence $(\zeta_i, Z_i, C_{t|_i})_{i=1}^n$ is $\psi$-admissible and therefore lies in $\eee_{K, \ee}(T.\ddd.\kkk, \mathbb{P})$.  Thus there exists $x^*_t\in K$ and a sequence $(x^t_i)_{i=1}^{|t|}\in \prod_{i=1}^{|t|}(B_X\cap Z_i\cap C_{t|_i})$ such that $$\text{Re\ }x^*_t\Bigl(\sum_{s\preceq t}\mathbb{P}(s) x^t_{|s|}\Bigr)\geqslant \ee.$$  Letting $x_{(s,t)}=x^t_{|s|}$ finishes the proof.

\end{proof}

Given $B$-trees $S,T$, we say a pair of functions $\theta:S.\ddd\to T.\ddd$, $e:MAX(S)\to MAX(T)$ is an \emph{extended pruning} provided it is monotone, if $s=(\zeta_i, Z_i)_{i=1}^m$ and $\theta(s)=(\mu_i, W_i)_{i=1}^n$, $W_n\subset Z_m$, and for any $(s,t)\in \Pi(S.\ddd)$, $\theta(s)\preceq e(t)$. We will write $(\theta, e):S.\ddd\to T.\ddd$ to denote an extended pruning.  

\begin{lemma} For any ordinal $\gamma>0$ and any finite subset $P_1, \ldots, P_n$ of $MAX(\ttt_\gamma.\ddd)$, there exist an extended pruning $(\theta, e):\ttt_\gamma.\ddd\to \ttt_\gamma.\ddd$ and $1\leqslant i\leqslant n$ such that $e(MAX(\ttt_\gamma.\ddd))\subset P_i$.  
\label{doug loves musicals}
\end{lemma}

\begin{proof} It was shown in \cite{proximity} that for any $0<\xi\leqslant \gamma$, there exists a function $\phi:\ttt_\xi\to \ttt_\gamma$ such that for any $\varnothing\prec s\preceq s_1\in \ttt_\xi$, $\phi(s)\prec \phi(s_1)$.   From this we easily deduce that for any $0<\xi\leqslant \gamma$, there exists an extended pruning $(\theta, e):\ttt_\xi.\ddd\to \ttt_\gamma.\ddd$.  Indeed, we first note that the function $\varphi:\ttt_\xi\to \ttt_\gamma$ given by $\varphi(s)=\phi(s)|_s$ is well-defined and still has the property that for any $\varnothing\prec s\preceq s_1\in \ttt_\xi$, $\varphi(s)\prec \varphi(s_1)$, and $\varphi$ preserves lengths.   We may then define $\theta((\zeta_i, Z_i)_{i=1}^n)=(\mu_i, Z_i)_{i=1}^n$, where $\varphi((\zeta_i)_{i=1}^n)=(\mu_i)_{i=1}^n$.    Then for every $t\in MAX(\ttt_\xi)$, let $e(t)$ be any maximal extension of $\theta(t)$, at least one of which exists by well-foundedness.

Recall that $\ttt_1.\ddd=\{(1,Z): Z\in \ddd\}$.  There exists $1\leqslant i\leqslant n$ such that the set $M=\{Z: (1,Z)\in P_i\}$ is cofinal in $\ddd$.    This means that for any $Z\in \ddd$, there exists $W_Z\in M$ such that $W_Z\leqslant Z$ and we may let $\theta((1,Z))=e((1,Z))=(1, W_Z)$.  Then $e(MAX(\ttt_1.\ddd))\subset P_i$. 

Next, suppose $\gamma$ is a limit ordinal and the result holds for all $\xi<\gamma$.  Recall that $\ttt_\gamma.\ddd=\cup_{\xi<\gamma}\ttt_{\xi+1}.\ddd$, and this is a disjoint union.    For every $\xi<\gamma$, there exist an extended pruning $(\theta_\xi, e_\xi):\ttt_{\xi+1}.\ddd\to \ttt_{\xi+1}.\ddd$ and $1\leqslant i_\xi\leqslant n$ such that $e_\xi(MAX(\ttt_{\xi+1}.\ddd))\subset P_{i_\xi}$.    There exists $1\leqslant i\leqslant n$ such that $M=\{\xi<\gamma: i_\xi=i\}$ has supremum $\gamma$.  For every $\xi<\gamma$, fix $\eta_\xi\in M$ with $\xi<\eta_\xi$ and an extended pruning $(\theta'_\xi, e'_\xi):\ttt_{\xi+1}.\ddd\to \ttt_{\eta_\xi+1}.\ddd$, as we may by the first paragraph of the proof.    Let $\theta|_{\ttt_{\xi+1}.\ddd}=\theta_{\eta_\xi+1}\circ \theta'_\xi$ and $e|_{MAX(\ttt_{\xi+1}.\ddd)}=e_{\eta_\xi+1}\circ e'_\xi$.   Then $e(MAX(\ttt_\gamma.\ddd))\subset P_i$.  

Next, assume the result holds for an ordinal $\xi>0$ and $\gamma=\xi+1$.   For $Z\in \ddd$, identifying $\{(\gamma,Z)\cat t: t\in \ttt_\xi.\ddd\}$ with $\ttt_\xi.\ddd$, we may find an extended pruning $(\theta_Z, e_Z):\ttt_\xi.\ddd\to \ttt_\xi.\ddd$ and $1\leqslant i_Z\leqslant n$ such that $\{(\gamma, Z)\cat e_Z(t): t\in MAX(\ttt_\xi.\ddd)\}\subset P_{i_Z}$.  There exists $1\leqslant i\leqslant n$ such that $M=\{Z\in\ddd: i_Z=i\}$ is cofinal in $\ddd$.   For $Z\in \ddd$, fix $W_Z\in M$ such that $W_Z\leqslant Z$ and define $$\theta((\gamma,Z))=(\gamma,W_Z),$$ $$\theta((\gamma,Z)\cat t)=(\gamma, W_Z)\cat \theta_{W_Z}(t),\hspace{5mm} t\in \ttt_\xi.\ddd,$$ $$e((\gamma,Z)\cat t)=(\gamma,W_Z)\cat e_{W_Z}(t),\hspace{5mm} t\in MAX(\ttt_\xi.\ddd).$$  This is an extended pruning with $e(MAX(\ttt_\gamma.\ddd))\subset P_i$.

\end{proof}

\begin{lemma}  Fix an ordinal $\xi>0$.  Suppose that $T$ is a well-founded, non-empty $B$-tree with $o(T)\geqslant \xi$ and $(x_{(s,t)})_{(s,t)\in \Pi(T.\ddd)}\subset B_X$ is normally weakly null.  Suppose also that for every $s\in T.\ddd$, $C_s$ is a norm compact subset of $X$ such that for every maximal extension $t$ of $s$, $x_{(s,t)}\in C_s$.  Then for any $\delta>0$, there exists a collection $(x_t')_{t\in \ttt_\xi.\ddd}\subset B_X$ which is normally weakly null and an extended pruning $(\theta, e):\ttt_\xi.\ddd\to T.\ddd$ such that for every $(s,t)\in \Pi(\ttt_\xi.\ddd)$, $\|x_s'-x_{(\theta(s), e(t))}\|<\delta.$

\label{finish2}
\end{lemma}

\begin{proof} We induct on $\xi$.  First suppose $\xi=1$.   Recall that $\ttt_1.\ddd=\{(1, Z):Z\in \ddd\}$, so that $\Pi(\ttt_1.\ddd)=\{((1,Z), (1,Z)):Z\in \ddd\}$.   Fix any $\zeta\in R_T$, as we may, since $o(T)\geqslant 1$.    For every $Z\in \ddd$, fix a maximal extension $t_Z$ of $(\zeta, Z)$.  Let $\theta((1,Z))=(\zeta, Z)$, $e((1,Z))=t_Z$, and let $x'_{(1,Z)}=x_{\theta((1,Z)), e((1,Z))}$.  The conclusions are easily seen to be satisfied in this case with $\delta=0$.

The limit ordinal case is trivial, since $\ttt_\xi.\ddd=\cup_{\zeta<\xi}\ttt_{\zeta+1}.\ddd$ is an incomparable union.  

Assume $\gamma>0$, the statement holds for $\gamma$, and $\xi=\gamma+1$.   Fix any $\zeta$ such that $(\zeta)\in T^\gamma$. Let $S$ denote those non-empty sequences $u$ such that $(\zeta)\cat s\in T$.  Fix $Z\in \ddd$.  Since $(\zeta, Z)\in T^\gamma$, $o(S.\ddd)\geqslant \gamma$ and $(x_{((\zeta, Z)\cat s, (\zeta, Z)\cat t)})_{(s,t)\in \Pi(S_Z.\ddd)}$ is normally weakly null.  Applying the inductive hypothesis to this collection and the sets $(C_{((\zeta, Z)\cat s, (\zeta,Z)\cat t)})_{(s,t)\in \Pi(S.\ddd)}$, we deduce the existence of a normally weakly null collection $(x^Z_{(s,t)})_{(s,t)\in \Pi(\ttt_\gamma.\ddd)}\subset B_X$ and an extended pruning $(\theta_Z, e_Z):\ttt_\zeta.\ddd\to S.\ddd$ such that for every $(s,t)\in \Pi(\ttt_\zeta.\ddd)$, $$\|x^Z_{(s,t)}-x_{((\zeta, Z)\cat \theta_Z(s), (\zeta, Z)\cat e_Z(t))}\|<\delta.$$   Next, let $(v_i)_{i=1}^n$ be a finite $\delta/2$-net of $C_{(\zeta,Z)}$.  Then if $$P_i=\{t\in MAX(\ttt_\gamma.\ddd): \|v_i-x_{((\zeta, Z), (\zeta, Z)\cat e_Z(t))}\|< \delta/2\},$$  by Lemma \ref{doug loves musicals}, there exists an extended pruning $(\theta'_Z, e'_Z):\ttt_\gamma.\ddd\to \ttt_\gamma.\ddd$ and $1\leqslant i_Z\leqslant n$ such that $e'_Z(MAX(\ttt_\gamma.\ddd))\subset P_{i_Z}$.  Fix $t_0\in MAX(\ttt_\gamma.\ddd)$ and let  $$x'_{(\xi, Z)}=x_{((\zeta, Z), (\zeta, Z)\cat e_Z\circ e'_Z(t_0))}, \hspace{5mm} x'_{((\xi, Z)\cat\theta_Z\circ\theta_Z'(t))},$$ $$\theta((\xi, Z))=(\zeta, Z), \hspace{5mm} \theta((\xi, Z)\cat t)= (\zeta, Z)\cat \theta_Z\circ \theta_Z'(t),$$  $$e((\xi, Z)\cat t)= e_Z\circ e'_Z(t).$$

\end{proof}

\begin{rem}\upshape Let $\nnn$ denote any weak neighborhood basis at $0$ in $X$.   Given a non-empty $B$-tree $T$, let us say that $(x_t)_{t\in T.\nnn}\subset B_X$ is \emph{usually weakly null} if for every $t=(\zeta_i, U_i)_{i=1}^n\in T.\nnn$, $x_t\in U_n$.   Note that for any $\delta>0$, there exist functions $\rho:\ddd\to \nnn$ and $\varrho:\nnn\to \ddd$ such that for any $Z\in \ddd$ and $U\in \nnn$, $B_Z\subset \rho(Z)\cap B_X$ and for any $x\in U\cap B_X$, there exists $y\in \subset B_{\varrho(U)}$ with $\|x-y\|<\delta$.  For $\ee>0$ and $\varnothing\neq K\subset X^*$, let $\hhh^K_\ee$ denote the empty sequence together with those sequences $(x_i)_{i=1}^n\in B_X^{<\nn}$ such that there exists $x^*\in K$ such that for every $1\leqslant i\leqslant n$, $\text{Re\ }x^*(x_i)\geqslant \ee$.   The main theorem of \cite{C} is the existence of a constant $c>0$ such that \begin{enumerate}[(i)]\item if there exists a usually weakly null $(x_t)_{t\in \ttt_{\omega^\xi}.\nnn}\subset B_X$ such that for every $t\in \ttt_{\omega^\xi}.\nnn$, $(x_{t|_i})_{i=1}^{|t|}\in \hhh^K_\ee$, then $Sz(K, \ee_1)>\omega^\xi$ for every $0<\ee_1<\ee$, and  \item if $Sz(K, c\ee)>\omega^\xi$, there exists a usually weakly null $(x_t)_{t\in \ttt_{\omega^\xi}.\nnn}\subset B_X$ such that for every $t\in \ttt_{\omega^\xi}.\nnn$, $(x_{t|_i})_{i=1}^{|t|}\in \hhh^K_\ee$.  \end{enumerate}

This combined with the existence of the functions $\rho, \varrho$ above, we deduce that \begin{enumerate}[(i)]\item if there exists a normally weakly null $(x_t)_{t\in \ttt_{\omega^\xi}.\ddd}\subset B_X$ such that for every $t\in \ttt_{\omega^\xi}.\ddd$, $(x_{t|_i})_{i=1}^{|t|}\in \hhh^K_\ee$, then $Sz(K, \ee_1)>\omega^\xi$ for every $0<\ee_1<\ee$, and \item for any $c'>c$, if $Sz(K, c'\ee)>\omega^\xi$, then there exists a normally weakly null $(x_t)_{t\in \ttt_{\omega^\xi}.\ddd}\subset B_X$ such that for every $t\in \ttt_{\omega^\xi}.\ddd$, $(x_{t|_i})_{i=1}^{|t|}\in \hhh^K_\ee$.    \end{enumerate}

From this, it follows that if $Sz(K)>\omega^\xi$, then there exists $\ee>0$ such that Player II has a winning strategy in the $\eee_{K, \ee}(\Gamma_\xi.\ddd.\kkk, \mathbb{P}_\xi)$ game.  Indeed, there exists $\ee>0$ such that $Sz(K, 2c\ee)>\omega^\xi$, and a normally weakly null $(x_t)_{t\in \ttt_{\omega^\xi}.\ddd}\subset B_X$ such that for every $t\in \ttt_{\omega^\xi}.\ddd$, $(x_{t|_i})_{i=1}^{|t|}\in \hhh^K_\ee$.  Since there exists a length-preserving, monotone $\theta:\Gamma_\xi\to \ttt_{\omega^\xi}$, we may let $\phi:\Gamma_\xi.\ddd\to \ttt_{\omega^\xi}.\ddd$ be given by $\phi((\zeta_i, Z_i)_{i=1}^n)=(\mu_i, Z_i)_{i=1}^n$, where $(\mu_i)_{i=1}^n=\phi((\zeta_i)_{i=1}^n)$.  By relabeling, we may assume we have a normally weakly null $(x_t)_{t\in\Gamma_\xi.\ddd}\subset B_X$ such that for every $t\in \Gamma_\xi.\ddd$, $(x_{t|_i})_{i=1}^{|t|} \in \hhh^K_\ee$. We define a winning strategy $\psi$ for Player II in the $\eee_{K, \ee}(\Gamma_\xi.\ddd.\kkk, \mathbb{P}_\xi)$ game.   Let $\psi(\varnothing, (\zeta, Z))=\{x_{(\zeta, Z)}\}$ and $\psi((\zeta_i, Z_i, C_i)_{i=1}^n, (\zeta_{n+1}, Z_{n+1}))=\{x_{(\zeta_i, Z_i)_{i=1}^{n+1}}\}$ for any compact sets $C_1, \ldots, C_n$.  Fix $t=(\zeta_i, Z_i, C_i)_{i=1}^n\in MAX(\Gamma_\xi.\ddd.\kkk)$ which is $\psi$-admissible, let $s=(\zeta_i, Z_i)_{i=1}^n$, and note that for each $1\leqslant i\leqslant n$, $C_i=\{x_{s|_i}\}$.  Then $(x_{s|_i})_{i=1}^n\in \prod_{i=1}^n (B_X\cap Z_i\cap C_i)$.  Since $(x_{s|_i})_{i=1}^n\in \hhh^K_\ee$, there exists $x^*\in K$ such that for every $1\leqslant i\leqslant n$, $\text{Re\ }x^*(x_i)\geqslant \ee$, and $$\text{Re\ }x^*\Bigl(\sum_{s\preceq t}\mathbb{P}_\xi(s)x_s\Bigr)\geqslant \ee.$$  Thus $t\in \eee_{K, \ee}(\Gamma_\xi.\ddd.\kkk, \mathbb{P}_\xi)$. 

The next corollary shows the converse of this fact.  

\label{remark}

\end{rem}

\begin{corollary} Suppose that $K\subset X^*$ is $w^*$-compact, $\ee>0$, and $\xi$ is an ordinal such that Player II has a winning strategy in the $\eee_{K, \ee}(\Gamma_\xi.\ddd.\kkk, \mathbb{P}_\xi)$ game. Then for any $0<\ee_1<\ee$, $Sz(K, \ee_1)>\omega^\xi$.

\end{corollary}

\begin{proof} Fix $\ee_1<\ee'<\ee$.   By Lemma \ref{finish}, we may fix a normally weakly null $(x_{(s,t)})_{(s,t)\in \Pi(\Gamma_\xi.\ddd)}\subset B_X$, $(x^*_t)_{t\in MAX(\Gamma_\xi.\ddd)}\subset K$, and $(C_s)_{s\in \Gamma_\xi.\ddd}\subset \kkk$ such that for every $t\in MAX(\Gamma_\xi.\ddd)$, $$\text{Re\ }x^*_t\Bigl(\sum_{s\preceq t}\mathbb{P}_\xi(s) x_{(s,t)}\Bigr)\geqslant \ee$$ and for every $s\in \Gamma_\xi.\ddd$ and every maximal extension $t$ of $s$, $x_{(s,t)}\in C_s$.  Fix $R>0$ such that $K\subset RB_{X^*}$ and define the function $$f:\Pi(\Gamma_\xi.\ddd)\to [-R,R]$$ by $f(s,t)=\text{Re\ }x^*_t(x_{(s,t)})$.   For every $t\in MAX(\Gamma_\xi.\ddd)$, $$\sum_{s\preceq t}\mathbb{P}_\xi(s)f(s,t)=\text{Re\ }x^*_t\Bigl(\sum_{s\preceq t}\mathbb{P}_\xi(s)x_{(s,t)}\Bigr) \geqslant \ee.$$  By \cite[Theorem $4.3$]{C2}, there exists an extended pruning $(\theta, e):\Gamma_\xi.\ddd\to \Gamma_\xi.\ddd$ such that for every $(s,t)\in \Pi(\Gamma_\xi.\ddd)$, $\text{Re\ }x^*_{e(t)}(x_{(\theta(s), e(t))})=f(\theta(s), e(t)) \geqslant \ee'$.   Fix $\delta>0$ such that $R\delta <\ee'-\ee_1$.    We may apply Lemma \ref{finish2} with this $\delta$ to the collection $(x_{(\theta(s), e(t))})_{(s,t)\in \Pi(\Gamma_\xi.\ddd)}$ and $(C_{\theta(s)})_{s\in \Gamma_\xi.\ddd}$ to obtain another extended pruning $(\theta', e'):\ttt_{\omega^\xi}.\ddd\to \Gamma_\xi.\ddd$ and a normally weakly null collection $(x_s')_{s\in \ttt_{\omega^\xi}.\ddd}\subset B_X$ such that for every $s\in \ttt_{\omega^\xi}.\ddd$ and every maximal extension $t$ of $s$, $$\|x_s-x_{\theta\circ \theta'(s), e\circ e'(t)}\|<\delta.$$   Fix any maximal $t\in \ttt_{\omega^\xi}.\ddd$ and note that $x^*_{e\circ e'(t)}\in K\subset R B_{X^*}$.  For any $1\leqslant i\leqslant |t|$, $$\text{Re\ }x^*_{e\circ e'(t)}(x'_{t|_i}) \geqslant \text{Re\ }x^*_{e\circ e'(t)}(x_{\theta\circ \theta'(t|_i), e\circ e'(t)})-R\|x'_{t|_i}-x_{\theta\circ \theta'(t|_i), e\circ e'(t)}\| \geqslant \ee'- R \delta.$$ Since $\ee'-R\delta>\ee_1$, Remark \ref{remark}  guarantees that $Sz(K, \ee_1)>\omega^\xi$.

\end{proof}

\begin{corollary} Given an ordinal $\xi$ and a $w^*$-compact set $K\subset X^*$, $Sz(K)>\omega^\xi$ if and only if there exists $\ee>0$ such that Player II has a winning strategy in the $\eee_{K, \ee}(\Gamma_\xi.\ddd.\kkk, \mathbb{P}_\xi)$ game.  

\label{games}
\end{corollary}

\subsection{Applications to essentially bounded trees in $L_p(X)$}

We recall the following special case of the main theorem of \cite{AJO}.  

\begin{theorem}[\cite{AJO}] If $X$ is a separable Banach space not containing $\ell_1$, then $Sz(X)>\omega$ if and only if there exists a $B$-tree $B$ with $o(B)=\omega$ and a weakly null collection $(f_t)_{t\in B}\subset B_X$ such that for every $t\in B$ and $f\in \text{\emph{co}}(f_s:s\preceq t)$, $\|f\|\geqslant \ee$. 
\end{theorem}

It is easy to see that $Sz(X)=1$ if and only if $X$ has finite dimension.  It was shown in \cite{KOS} that any asymptotically uniformly smooth Banach space has Szlenk index not exceeding $\omega$, whence for any $1<p<\infty$, $Sz(L_p)=\omega$.  It is also easy to see that the Szlenk index is an isomorphic invariant, so that any Banach space isomorphic to $L_p$ has Szlenk index $\omega$.

Recall that for $1<p<\infty$, $L_p(X)$ denotes the Banach space of (equivalence classes of) Bochner integrable functions $f:[0,1]\to X$ such that $\int \|f\|^p <\infty$, where $[0,1]$ is endowed with Lebesgue measure.  We let $L_\infty(X)$ denote the $X$-valued strongly measurable functions which are essentially bounded.   It is well known and easy to see that for any subspace $Z$ of $X$, $L_p(X)/L_p(Z)$ is canonically isometrically isomorphic to $L_p(X/Z)$ by the operator $\Phi$ such that for each $\varpi\in [0,1]$, $\Phi(f+L_p(Z))(\varpi)=f(\varpi)+Z$.   Moreover, if $\dim X/Z<\infty$, $L_p(X/Z)$ is isomorphic to $L_p$ and therefore has Szlenk index $\omega$.   This means that for any $B$-tree $T$ with $o(T)\geqslant \omega$ and any weakly null collection $(\overline{f}_t)_{t\in T}\subset B_{L_p(X/Z)}$ and any $\delta>0$,  there exists $t\in T$ and a convex combination $\overline{f}$ of $(\overline{f}_s:s\preceq t)$ such that $\|f\|<\delta$.   This means that if $T$ is a $B$-tree with $o(T)\geqslant \omega$, $\dim X/Z<\infty$, $\delta>0$, and if $(f_t)_{t\in T}\subset  B_{L_p(X)}$ is a weakly null collection, there exists $t\in T$ and $f\in \text{co}(f_s:s\preceq t)$ such that $\|f\|_{L_p(X)/L_p(Z)}<\delta$.   Indeed, we simply let $\overline{f}_t=f_t+L_p(Z)$ and use the previous fact, noting that $(\overline{f}_t)_{t\in T}$ is still weakly null and contained in $B_{L_p(X)/L_p(Z)}$ and using the isometric identification of $L_p(X)/L_p(Z)$ and $L_p(X/Z)\approx L_p$.   Finally, if $f\in CB_{L_\infty(X)}$ and $\|f\|_{L_p(X)/L_p(Z)}<\delta$, then there exists a simple function $g\in 2CB_{L_\infty(Z)}$ such that $\|f-g\|_{L_p(X)}<\delta$.   Indeed, we may first fix $h\in L_p(Z)$ such that $\|f-h\|_{L_p(X)}<\delta$ and, by density of simple functions in $L_p(Z)$, assume $h$ is simple.  Next, let $E=\{\varpi: \|h(\varpi)\|>2C\}$.  Note that there exists a subset $N$ of $E$ having measure zero such that for all $\varpi\in E\setminus N$, $$\|f(\varpi)\|\leqslant C \leqslant \|h(\varpi)\|-\|f(\varpi)\| \leqslant \|h(\varpi)-f(\varpi)\|.$$  Thus we deduce that \begin{align*} \|f-1_{E^C}h\|^p & = \int_E \|f\|^p + \int_{E^C} \|f-h\|^p \leqslant \int_E \|f-h\|^p+\int_{E^C}\|f-h\|^p <\delta^p.   \end{align*} Thus $g=1_{E^C}h$ is the simple function we seek.

 Fix $1<p<\infty$ and let $q$ be the conjugate exponent to $p$.  For a fixed $K\subset X^*$, let $M$ denote the $K$-valued, measurable simple functions in $L_q(X^*)$.  Recall that $L_q(X^*)$ is canonically isometrically included in $L_p(X)^*$ via the action $g(f)=\int g(\varpi)(f(\varpi))\rangle d\varpi$.     Finally, let $\ddd_0$ denote the subspaces of $L_p(X)$ having finite codimension in $L_p(X)$.

\begin{theorem} With $K$, $M$ as above, if $Sz(K)\leqslant \omega^\xi$, then for any $B$-tree $S$ with $o(S)\geqslant \omega^{1+\xi}$, any weakly null collection  of simple functions $(f_t)_{t\in S}\subset \frac{1}{2}B_{L_\infty(X)}$, and any $\ee>0$, there exist $t\in S$ and $f\in \text{\emph{co}}(f_s:s\preceq t)$ such that $\underset{h\in M}{\sup}\text{\emph{Re}}\int hf\leqslant\ee.$

\label{work}
\end{theorem}

\begin{proof} Fix $R>0$ such that $K\subset RB_{X^*}$ and note that $M\subset R B_{L_p(X)^*}$.   By Proposition \ref{determined}, the $\eee=\eee_{K, \ee/2}(\Gamma_\xi.\ddd.\kkk, \mathbb{P}_\xi)$ game on $\Gamma_\xi.\ddd.\kkk$ is determined.  Since $Sz(K)\leqslant \omega^\xi$, Corollary \ref{games} implies that Player II cannot have a winning strategy, and therefore Player I has a winning strategy.  Fix a winning strategy $\varphi$ for Player I.  Define $m:\Gamma_\xi.\ddd\to [0, \omega^\xi)$ by letting $m(t)=\max\{\gamma<\omega^\xi: t\in (\Gamma_\xi.\ddd)^\gamma\}$.   

We next define several sequences recursively.  Let $\varphi(\varnothing)=(\zeta_1, Z_1)$.  Let $\gamma_1=m(\zeta_1, Z_1)$.  Note that since $\gamma_1<\omega^\xi$ and $o(S)\geqslant \omega^{1+\xi}$, $o(S^{\omega \gamma_1})\geqslant \omega$ and $(f_t)_{t\in S^{\omega \gamma_1}}\subset \frac{1}{2}B_{L_\infty(X)}$ is normally weakly null.  By the remarks in the paragraphs preceding the statement of the theorem, there exist $s_1\in S^{\omega \gamma_1}$, a convex combination $f_1$ of $(f_t:t\preceq s_1)$, and a simple function $g_1\in B_{L_\infty(Z_1)}$ such that $\|f_1-g_1\|_{L_p(X)}<\ee/2R$. By redefining $g_1$ on a set of measure zero, we may assume $\text{range}(g_1)\subset B_{Z_1}$ is finite.  Let $C_1=\text{range}(g_1)\subset B_X$.   

Next, suppose that for each $1\leqslant i\leqslant n$, $\zeta_i$, $Z_i$, $C_i$, $s_i$, $\gamma_i$, $f_i$, $g_i$ have been defined to have the following properties: \begin{enumerate}[(i)]\item $\varphi((\zeta_j, Z_j, C_j)_{j=1}^{i-1})=(\zeta_i, Z_i)$, \item $\|f_i-g_i\|_{L_p(X)}<\ee/2R$, \item $C_i=\text{range}(g_i)\subset B_{Z_i}$ is finite, \item $\gamma_i=m((\zeta_j, Z_j)_{j=1}^i)$, \item $f_i\in \text{co}(f_s:s_{i-1}\prec s\preceq s_i)$, (where $s_0=\varnothing$). \end{enumerate}

If $(\zeta_i, Z_i)_{i=1}^n$ is maximal in $\Gamma_\xi.\ddd$, we have completed the recursive construction. Suppose that $(\zeta_i, Z_i)_{i=1}^n$ is not maximal in $\Gamma_\xi.\ddd$. Let $\varphi((\zeta_i, Z_i, C_i)_{i=1}^n)=(\zeta_{n+1}, Z_{n+1})$.   Let $\gamma_{n+1}=m((\zeta_i, Z_i)_{i=1}^{n+1})$.   Let $U$ denote those non-empty sequences $s$ such that $s_n\cat s\in S^{\omega \gamma_{n+1}}$.  Applying the remarks in the paragraphs preceding the proof to the collection $(f_{s_n\cat s})_{s\in U}$, we deduce the existence of $s_{n+1}\in S^{\omega \gamma_{n+1}}$, $f_{n+1}\in \text{co}(f_s: s_n\prec s\preceq s_{n+1})$, and $g_{n+1}\in B_{L_\infty(Z_{n+1})}$ such that $\|f_{n+1}-g_{n+1}\|_{L_p(X)}<\ee/2R$.   Here we have used that since $s_n\in S^{\omega \gamma_n}$ and $\gamma_{n+1}<\gamma_n$, $o(U)\geqslant \omega$.  By redefining $g_{n+1}$ on a set of measure zero, we may assume $\text{range}(g_{n+1})\subset B_{Z_{n+1}}$ is finite.  Let $C_{n+1}=\text{range}(g_{n+1})$.

Since $\Gamma_\xi.\ddd$ is well-founded, this process must eventually terminate.    Assume that the process terminates with the sequence $(\zeta_i, Z_i)_{i=1}^n\in MAX(\Gamma_\xi.\ddd)$, the sequences $s_i$, and the functions $f_i, g_i$.  By our choices, $(\zeta_i, Z_i, C_i)_{i=1}^n$ is $\varphi$-admissible, and therefore not a member of $\eee_{K, \ee/2}(\Gamma_\xi.\ddd.\kkk, \mathbb{P}_\xi)$.  This means  that for any $(x_i)_{i=1}^n\in \prod_{i=1}^n (B_X\cap Z_i\cap C_i)$ and for all $x^*\in  K$, $\text{Re\ }x^*\Bigl(\sum_{i=1}^n \mathbb{P}_\xi(t|_i)x_i\Bigr)< \ee/2.$  But for any $\varpi\in [0,1]$, $(g_i(\varpi))_{i=1}^n\in \prod_{i=1}^n (B_X\cap Z_i\cap C_i)$, whence for any $x^*\in K$, $\text{Re\ }x^*\Bigl(\sum_{i=1}^n \mathbb{P}_\xi(t|_i)g_i(\varpi)\Bigr)<\ee/2$.    Then with $g=\sum_{i=1}^n \mathbb{P}_\xi(t|_i)g_i$ and $h\in M$, $\text{Re}\int hg \leqslant \ee/2$, whence $$\sup_{h\in M}\text{Re}\int hg\leqslant \ee/2.$$   Let $f=\sum_{i=1}^n \mathbb{P}_\xi(t|_i)f_i\in \text{co}(f_s:s\preceq s_n)$ and note that $$\|f-g\|_{L_p(X)} \leqslant \sum_{i=1}^n \mathbb{P}_\xi(t|_i)\|f_i-g_i\|_{L_p(X)}<\ee/2R.$$   Since $M\subset RB_{L_p(X)^*}$, it follows that $$\sup_{h\in M}\text{Re}\int hf \leqslant \sup_{h\in M}\text{Re}\int hg + R\|f-g\|_{L_p(X)}\leqslant \ee.$$

\end{proof}

\section{The $w^*$-dentability index and a result of Lancien}

In this section, we again fix $1<p<\infty$ and let $q$ be the conjugate exponent to $p$.  Let $\mathcal{W}$ be a $w^*$-neighborhood basis at $0$ in $L_p(X)^*$.  The following was shown in \cite{C} in the case that $L$ is $w^*$-compact.  However, the proof given there does not depend upon the $w^*$-compactness of $L$. For the remainder of the section, $K\subset X^*$ will be a fixed $w^*$-compact, non-empty set and $M$ will denote the subset of $L_q(X^*)\subset L_p(X)^*$ consisting of all $K$-valued, measurable simple functions.  

\begin{proposition} For an ordinal $\xi$, if $h\in s^\xi_{2\ee}(L)$, there exists a collection $(h_t)_{t\in (\ttt_\xi\cup \{\varnothing\}).\mathcal{W}}\subset L$ such that $h_\varnothing=h$ and for every $t\in \ttt_\xi.\mathcal{W}$, if $t=(\zeta_i, V_i)_{i=1}^n$, $\|h_t-h_{t^-}\|_{L_p(X)^*}>\ee$ and $h_t-h_{t^-}\in V_n$.

\end{proposition}

A collection $(h_t)_{t\in (\ttt_\xi\cup \{\varnothing\}.\mathcal{W})}$ satisfying the condition that for any $ t\in \ttt_\xi.\mathcal{W}$, if $t=(\zeta_i, V_i)_{i=1}^n$, then $h_t-h_{t^-}\in V_n$ will be called \emph{normally} $w^*$-\emph{closed}.   A collection such that for any $t\in \ttt_\xi.\mathcal{W}$, $\|h_t-h_{t^-}\|>\ee$ will be called $\ee$-\emph{separated}. 

Although it was not stated in this way, the following theorem was shown in \cite{L}. Since the statement of this theorem differs significantly from the statement in \cite{L}, we will sketch the statement here for completeness.

\begin{theorem}\cite[Lemma1]{L} Suppose that $K$ is convex. If $n\in \nn$ and $x_1^*,\ldots, x_n^*\in d_{2\ee}^\xi(K)$, then $\sum_{i=1}^n x^*_i 1_{[\frac{i-1}{n}, \frac{i}{n})}\in s_\ee^\xi(M)$.   
\label{Lancien thing}
\end{theorem}

\begin{proof}[Sketch] Given $K_0\subset X^*$ and $L_0\subset L_p(X)^*$, let us say the pair $(K_0, L_0)$ is \emph{nice} provided that $K_0$ is $w^*$-compact, convex, and symmetric, and for any $n\in \nn$ and $x_1^*, \ldots, x_n^*\in K_0$, $\sum_{i=1}^n x^*_i 1_{[\frac{i-1}{n}, \frac{i}{n})}\in L_0$. Of course, if $(K_0, L_0)$ is nice and $K_0\neq \varnothing$, then $L_0\neq \varnothing$.   We claim that if $(K_0, L_0)$ is nice, then for any $\ee>0$, the pair $(d_{2\ee}(K_0), s_\ee(L_0))$ is nice.   An easy induction then yields that for any ordinal $\xi$, the pair $(d_{2\ee}^\xi(K_0), s_\ee^\xi(L_0))$ is nice, whence $Dz(K_0, 2\ee)\leqslant Sz(L_0, \ee)$.  We obtain Theorem \ref{Lancien thing} by noting that $(K, M)$ is nice.    

We prove the claim that $(d_{2\ee}(K_0), s_\ee(L_0))$ is nice, assuming $(K_0, L_0)$ is nice.  Of course, $d_{2\ee}(K_0)$ is $w^*$-compact, convex, and symmetric.  Fix $n\in \nn$ and $x_1^*, \ldots, x_n^*\in d_{2\ee}(K_0)$. Let $f=\sum_{i=1}^n x_i^*1_{[\frac{i-1}{n}, \frac{i}{n})}\in L_0$.  Let $V$ be a $w^*$-open neighborhood of $f$.   It follows by the Hahn-Banach theorem that for each $1\leqslant i\leqslant n$, $x^*_i$ lies in the $w^*$-closed, convex hull of $K_0\setminus (x_i^*+\ee B_{X^*})$.  Then there exist $k\in \nn$ and $(x_{ij}^*)_{i=1, j=1}^{n,k}\subset K$ such that $$g=\sum_{i=1}^n \Bigl(k^{-1}\sum_{j=1}^k x_{ij}^*\Bigr)1_{[\frac{i-1}{n}, \frac{i}{n})}\in V.$$  

For each $l\in \nn$, let $$\psi_l=\sum_{i=1}^n\sum_{j=1}^k \sum_{m=1}^l x^*_{ij} 1_{I_{ijm}},$$ where $$I_{ijm}=\Bigl[\frac{i-1}{n}+\frac{m-1}{nl}+\frac{j-1}{nlk}, \frac{i-1}{n}+\frac{m-1}{nl}+\frac{j}{nlk}\Bigr).$$  Note that $\psi_l\underset{w^*}{\to} g$, whence $\psi_l\in V$ for sufficiently large $l\in \nn$.   Since $(K_0, L_0)$ is nice, $\psi_l\in L_0$ for all $l\in \nn$.  Also, for any $\varpi\in [0,1]$, $\|f(\varpi)-\psi_l(\varpi)\|>\ee$, whence $\|f-\psi_l\|_{L_p(X)^*}>\ee$.  This shows that $f\in s_\ee(L_0)$.

\end{proof}

We remark that if $h\in L_q(X^*)$ is a simple function such that $\|h\|_{L_q(X^*)}>\ee>0$ and $\|h\|_{L_\infty(X^*)}\leqslant C$, there exists a simple function $f\in B_{L_p(X)}$ with $\|f\|_{L_\infty(X)}\leqslant C^{q-1}/\ee^{q-1}$ and $\int hf>\ee$.   Indeed, write $h=\sum_{i=1}^n x^*_i 1_{F_i}$ with $F_i$ pairwise disjoint and measurable.  Fix $0<\rho<1$ such that $\rho \|h\|_{L_q(X^*)} >\ee$. For each $1\leqslant i\leqslant n$, fix $x_i\in S_X$ such that $x^*_i(x_i)>\rho \|x_i^*\|$.   Then $f=\|h\|_{L_q(X^*)}^{1-q}\sum_{i=1}^n \|x^*_i\|^{q-1} x_i 1_{F_i}$ has the indicated properties by familiar computations.

\begin{lemma}   Suppose that $R\geqslant 1$ is such that $K\subset RB_{X^*}$.  Assume $(h_t)_{t\in (\ttt_\xi\cup \{\varnothing\}).\mathcal{W}}\subset M$ is normally $w^*$-closed and $\ee$-separated for some $\ee\in (0,1)$.  Then there exist a function $\theta:\ttt_\xi.\ddd\to \ttt_\xi.\mathcal{W}$ and a  weakly null collection $(f_t)_{t\in \ttt_\xi.\mathcal{N}}\subset \frac{1}{2}B_{L_\infty(X)}$ such that for any $\varnothing \prec s\preceq t$, $$\text{\emph{Re}}\int h_{\theta(t)}f_s \geqslant \frac{\ee^q}{3 \cdot 2^q R^{q-1}}.$$  Here, $\mathcal{N}$ denotes the directed set of convex, weakly open neighborhoods of $0$ in $L_p(X)$.

\end{lemma}

\begin{proof} We will need the following claim.  

\begin{claim} If $(h_t)_{t\in (\ttt_\xi\cup \{\varnothing\}).\mathcal{W}}\subset M$ is normally $w^*$-closed and $\ee$-separated, then for any sequence $(\ee_n)_{n=0}^\infty$ of positive numbers, there exists a monotone, length-preserving function $\theta:\ttt_\xi.\ddd_0\to \ttt_\xi.\mathcal{W}$ and a collection $(g_t)_{t\in \ttt_\xi.\mathcal{N}}\subset \frac{2^{q-1} R^{q-1}}{\ee^{q-1}}B_{L_\infty(X)}$ such that for every $s\in \ttt_\xi.\mathcal{N}$, $\text{\emph{Re}}\int h_{\theta(t)}g_s > \ee/2-\ee_0$ and such that for any $\varnothing\prec s\prec t$, $|\int (h_{\theta(t)}-h_{\theta(t^-)})g_s|<\ee_{|t|}$, and if $s=(\zeta_i, U_i)_{i=1}^n$, $g_s\in U_n$.  

\end{claim}

We first assume the claim and finish the proof.  We apply the claim with some sequence $(\ee_n)_{n=0}^\infty$ such that $\ee/2-\sum_{n=0}^\infty \ee_n>\ee/3$.  Fix any $t\in \ttt_\xi.\ddd_0$ and let $\varnothing\prec s\preceq t$.   Then \begin{align*} \text{Re}\int h_{\theta(t)}g_s & \geqslant \text{Re}\int h_{\theta(s)}g_s - \sum_{s\prec u\preceq t}|\int (h_{\theta(u)}-h_{\theta(u^-)})g_s| > \ee/2-\ee_0-\sum_{n=|s|+1}^{|t|}\ee_n \\ & > \ee/2-\sum_{n=0}^\infty \ee_n>\ee/3.\end{align*} From this it follows that for any convex combination $g$ of $(g_s:s\preceq t)$, $\text{Re}\int h_{\theta(t)}g >\ee/3$.   Letting $f_t=\frac{\ee^{q-1}}{2^q R^{q-1}}g_t$ gives the desired collection.  Since $\ee^{q-1}/2^q R^{q-1}\in (0,1)$, if $s=(\zeta_i, U_i)_{i=1}^n$, $f_s\in U_n$, since $U_n$ is a convex weak neighborhood of $0$ and $g_s\in U_n$.  This condition guarantees that the collection $(f_t)_{t\in \ttt_\xi.\mathcal{N}}$ is weakly null.

We return to the proof of the claim.  We define $g_s$ and $\theta(s)$ by induction on $|s|$, and we define $\theta$ to be monotone, length-preserving, and so that for any $s=(\zeta_i, U_i)_{i=1}^n$, $\theta(s)=(\zeta_i, V_i)_{i=1}^n$ for some $V_1, \ldots, V_n\in \mathcal{W}$.  Note that these properties together imply that if $\theta(s)=(\zeta_i, V_i)_{i=1}^n$ and $1\leqslant m\leqslant n$, $\theta(s|_m)=(\zeta_i, V_i)_{i=1}^m$.

 Suppose $(h_t)_{t\in (\ttt_\xi\cup \{\varnothing\}).\mathcal{W}}\subset M$ is as in the claim.  Fix some $(\zeta, U)\in \ttt_\xi.\mathcal{N}$.   For every $V\in \mathcal{W}$, $\|h_{(\zeta, V)}-h_\varnothing\|_{L_q(X^*)}>\ee$ and $h_{(\zeta, V)}-h_\varnothing$ is a simple function with $\|h_{(\zeta, V)}-h_\varnothing\|_{L_\infty(X^*)}\leqslant 2R$.   By the remarks preceding the lemma, there exists a simple function $j_V\in B_{L_p(X)}\cap \frac{2^{q-1}R^{q-1}}{\ee^{q-1}}$ such that $\text{Re}\int (h_{(\zeta, V)}-h_\varnothing)j_V>\ee$.  By \cite[Lemma $3.3$]{CD}, for each $U\in \mathcal{N}$, there exist $V^U_1, V^U_2\in \mathcal{W}$ such that $$\text{Re}\int h_{(\zeta, V^U_2)}\Bigl(\frac{j_{V_2^U}-j_{V_1^U}}{2}\Bigr) >\ee/2-\ee_0$$  and $$\frac{j_{V_2^U}-j_{V_1^U}}{2}\in U.$$   We let $g_{(\zeta, U)}= \frac{j_{V^U_2}-j_{V_1^U}}{2}$ and $\theta((\zeta, U))=(\zeta, V^U_2)$.    
 
Now suppose that for some $s=s_1\cat (\eta, W)\in \ttt_\xi.\mathcal{N}$ with $s_1\neq \varnothing$, and for every $\varnothing\prec u\preceq s_1$, $g_u$ and $\theta(u)$ have been defined to have the indicated properties.  Let $t=\theta(s_1)$.  For every $V\in \mathcal{W}$, $\|h_{t\cat (\eta, V)}-h_t\|>\ee$, $\|h_{t\cat (\eta, V)}-h_t\|_{L_\infty(X^*)}\leqslant 2R$, and the function $h_{t\cat (\eta,V)}-h_t$ is simple, whence there exists a simple function $i_V\in B_{L_p(X)}$ with $\|i_V\|_{L_\infty(X)} \leqslant 2^{q-1}R^{q-1}/\ee^{q-1}$ such that $\text{Re}\int (h_{t\cat (\eta, V)}-h_t)i_V> \ee$.   Again using \cite[Lemma $3.3$]{CD}, there exist $V^W_1, V^W_2\in \mathcal{W}$ such that $$\text{Re}\int h_{t\cat(\eta, V^W_2)}\Bigl(\frac{i_{V_2^W}-i_{V_1^W}}{2}\Bigr) >\ee/2-\ee_0,$$ $$\frac{i_{V_2^W}-i_{V_1^W}}{2}\in W,$$ and $$V_2^W\subset \{h\in L_p(X)^*: (\forall \varnothing\prec u\preceq s_1)(|h(g_u)|<\ee_{|s|})\}.$$ We let $g_s= \frac{i_{V^W_2}-i_{V^W_1}}{2}$ and $\theta(s)=t \cat (\eta, V^W_2)$.  This finishes the construction, and the conclusions of the claim are easily verified.

\end{proof}

\begin{corollary} If $K$ is convex and  $Dz(K)>\omega^{1+\xi}$, then there exists a constant $\ee'>0$ and a  weakly null collection $(f_t)_{t\in \ttt_{\omega^{1+\xi}}.\mathcal{N}}\subset \frac{1}{2}B_{L_\infty(X)}$ such that for every $t\in \ttt_{\omega^{1+\xi}}.\ddd_0$ and every convex combination $f$ of $(f_s: s\preceq t)$, $$\underset{h\in M}{\sup} \text{\emph{Re}}\int hf \geqslant \ee'.$$

\label{corollary}
\end{corollary}

\section{Proof of the main theorem}

\begin{proof}[Proof of Theorem \ref{main theoremm}] Let $K$ denote the $w^*$-closed, convex, symmetrized hull of $K_0$.  By \cite[Theorem $1.5$]{C2}, $Sz(K)\leqslant \omega^\xi$.    If $Dz(K)>\omega^{1+\xi}$, there exists a constant $\ee'>0$ and a normally weakly null collection $(f_t)_{t\in \ttt_{\omega^{1+\xi}}.\mathcal{N}}\subset \frac{1}{2}B_{L_\infty(X)}$ as in the conclusoin of Corollary \ref{corollary}.  By Theorem \ref{work}, the existence of such a collection implies that $Sz(K)>\omega^\xi$.  It follows that $Dz(K)\leqslant \omega^{1+\xi}$, whence $Dz(K_0)\leqslant Dz(K)\leqslant \omega^{1+\xi}$. This gives $(i)$.        

For $(ii)$, note that if $K$ is convex, $Sz(K)=\omega^\xi$ for some ordinal $\xi$, or $Sz(K)=\infty$ if $K$ is not $w^*$-fragmentable.   In the first case, by $(i)$, we deduce that $Dz(K)\leqslant \omega^{1+\xi}=\omega Sz(K)$.  If $K$ is not $w^*$-fragmentable, it is not $w^*$-dentable, and $Dz(K)=\infty=\omega \infty=\omega Sz(K)$ by convention.  If $Sz(K)\geqslant \omega^\omega$, then either $Sz(K)=Dz(K)=\infty$ or $Sz(K)=\omega^\xi$ and $Dz(K)\leqslant \omega^{1+\xi}$ for an ordinal $\xi\geqslant \omega$.  But since $\xi\geqslant \omega$, $1+\xi=\xi$, and $Dz(K)\leqslant \omega^{1+\xi}=\omega^\xi=Sz(K)$.

\end{proof}

As we have already mentioned, for every $n\in \nn\cup \{0\}$, there exist a pair of Banach spaces $X_n$, $Y_n$ such that $Sz(X_n)=Dz(X_n)=\omega^n$ and $Dz(Y_n)=\omega Sz(Y_n)=\omega^{n+1}$, so that Theorem \ref{main theoremm} is sharp.

If $A:X\to Y$ is an operator, for any $1<p<\infty$, $A$ induces an operator $A_p:L_p(X)\to L_p(Y)$ such that for any $\varpi\in [0,1]$, $(A_p f)(\varpi)=A(f(\varpi))$.  Since $(A^*B_{Y^*}, (A_p)^*B_{L_p(Y)^*})$ is nice, Theorem \ref{Lancien thing} yields that $Dz(A)\leqslant Sz(A_p)$.    Thus a positive solution to the next question implies Theorem \ref{main theoremm}.  

\begin{question} For any operator $A:X\to Y$ and $1<p<\infty$, is it true that $Sz(A_p)\leqslant \omega Sz(A)$?   
\label{question}
\end{question} 

By \cite{HS}, Question \ref{question} has a positive answer when $A$ is an identity operator and $Sz(A)$ is countable.  It is possible to deduce using arguments similar to those in \cite{HS} that if $Sz(A)$ is countable, Question \ref{question} has a positive answer.  

A positive solution to the following question would imply a positive solution to Question \ref{question}.

\begin{question} For any operator $A:X\to Y$ and $1<p<\infty$, is it true that $Dz(A)=Sz(A_p)$?

\end{question}


\begin{thebibliography}{HD}

\normalsize
\baselineskip=17pt

\bibitem{AJO} D. Alspach, R. Judd, E. Odell. \emph{The Szlenk index and local $\ell_1$-indices}, Positivity, 9 (2005), no. 1,1-44. 

\bibitem{Brooker} P.A.H. Brooker, \emph{Asplund operators and the Szlenk index}, Operator Theory 68 (2012) 405-442.

\bibitem{proximity} R.M. Causey, \emph{Proximity to $\ell_p$ and $c_0$ in Banach spaces}, J. Funct. Anal., (2015) 10.1016/j.jfa.2015.10.001. 

\bibitem{C} R.M. Causey, \emph{An alternate description of the Szlenk index with applications}, Illinois J. Math.
Volume 59, (2015) no. 2, 359 390. 

\bibitem{C2} R.M. Causey, \emph{The Szlenk index of $C(K,X)$}, submitted.  

\bibitem{Cpower} R.M. Causey, \emph{Power type $\xi$-asymptotically uniformly smooth norms}, preprint. 

\bibitem{CD} R.M. Causey, S.J. Dilworth, \emph{$\xi$-asymptotically uniformly smooth, $\xi$-asymptotically uniformly convex, and $(\beta)$-opeators}, submitted. 

\bibitem{KOS} H. Knaust, E. Odell, Th. Schlumprecht, \emph{On asymptotic structure, the Szlenk index and UKK properties
in Banach spaces}, Positivity 3 (1999), 173-199.


\bibitem{GKL} G. Godefroy, N.J. Kalton, G. Lancien, \emph{Szlenk indices and uniform homeomorphisms}, Trans. Amer. Math. Soc. 353 (2001) 3895-3918.

\bibitem{HL} P. H\'{a}jek, G. Lancien, \emph{Various slicing indices on Banach spaces}, Mediterr. J. Math. 4(2007) 179-190.

\bibitem{HS} P. H\'{a}jek and Th. Schlumprecht, \emph{The Szlenk index of $L_p(X)$}, Bull. Lond. Math. Soc. 46 (2014)  no. 2, 415-424.

\bibitem{L0} G. Lancien, \emph{On the Szlenk index and the weak$^*$ dentability index}, Quarterly J. Math. Oxford, 47 (1996) 59-71. 

\bibitem{L1} G. Lancien, \emph{On uniformly convex and uniformly Kadec-Klee renormings}, Serdica Math. J. 21 (1995) 1-18.

\bibitem{L} G Lancien, \emph{A survey on the Szlenk index and some of its applications}, RACSAM Rev. R. Acad. Cienc. Exactas F'is. Nat. Ser. A Mat. 100 (2006), no. 1-2, 209-235.

\bibitem{LPR} G. Lancien, A. Proch\'{a}zka, M. Raja, \emph{Szlenk indices of convex hulls}, arXiv:1504.06997. 

\bibitem{R}  M. Raja, \emph{Dentability indices with respect to measures of non-compactness}, J. Funct. Anal. 253 (2007), 273-286.

\bibitem{Szlenk} W. Szlenk, \emph{The non existence of a separable reflexive Banach space universal for all separable reflexive Banach spaces}, Studia Math. 30 (1968), 53-61. 

\end{thebibliography}
\end{document}